\documentclass[12pt]{amsart}
\usepackage{amscd,amsmath,amsthm,amssymb,graphics}
\usepackage{slashbox}
\usepackage{pgf,tikz, pifont}

\usetikzlibrary{arrows}

\usepackage[all]{xy}
\usepackage{lineno}

\def\frk{\frak}               

\def\Phi{{\frk n}}
\def\Phi{{\frk N}}
\def\opn#1#2{\def#1{\operatorname{#2}}} 
\opn\chara{char} \opn\length{\ell} \opn\pd{pd} \opn\rk{rk}
\opn\projdim{proj\,dim} \opn\injdim{inj\,dim} \opn\rank{rank}
\opn\depth{depth} \opn\grade{grade} \opn\height{height}
\opn\embdim{emb\,dim} \opn\codim{codim}

\opn\Tr{Tr} \opn\bigrank{big\,rank}
\opn\superheight{superheight}\opn\lcm{lcm}
\opn\trdeg{tr\,deg}
\opn\reg{reg} \opn\lreg{lreg} \opn\ini{in} \opn\lpd{lpd}
\opn\size{size}\opn\bigsize{bigsize}
\opn\cosize{cosize}\opn\bigcosize{bigcosize}
\opn\sdepth{sdepth}\opn\sreg{sreg}
\opn\link{link}\opn\fdepth{fdepth}
\opn\index{index}
\opn\index{index}
\opn\indeg{indeg}
\opn\N{N}
\opn\SSC{SSC}
\opn\SC{SC}
\opn\lk{lk}
\opn\div{div} \opn\Div{Div} \opn\cl{cl} \opn\Cl{Cl}
\opn\Spec{Spec} \opn\Supp{Supp} \opn\supp{supp} \opn\Sing{Sing}
\opn\Ass{Ass} \opn\Min{Min}\opn\Mon{Mon} \opn\dstab{dstab} \opn\astab{astab}
\opn\Syz{Syz}
\opn\reg{reg}
\opn\Ann{Ann} \opn\Rad{Rad} \opn\Soc{Soc}
\opn\Im{Im} \opn\Ker{Ker} \opn\Coker{Coker} \opn\Am{Am}
\opn\Hom{Hom} \opn\Tor{Tor} \opn\Ext{Ext} \opn\End{End}\opn\Der{Der}
\opn\Aut{Aut} \opn\id{id}

\opn\nat{nat}
\opn\pff{pf}
\opn\Pf{Pf} \opn\GL{GL} \opn\SL{SL} \opn\mod{mod} \opn\ord{ord}
\opn\Gin{Gin} \opn\Hilb{Hilb}\opn\sort{sort}
\opn\initial{init}
\opn\ende{end}
\opn\height{height}
\opn\type{type}
\opn\aff{aff} \opn\con{conv} \opn\relint{relint} \opn\st{st}
\opn\lk{lk} \opn\cn{cn} \opn\core{core} \opn\vol{vol}
\opn\link{link} \opn\Link{Link}\opn\lex{lex}
\opn\gr{gr}

\def\pot#1#2{#1[\kern-0.28ex[#2]\kern-0.28ex]}

\opn\dirlim{\underrightarrow{\lim}}
\opn\inivlim{\underleftarrow{\lim}}
\def\Implies{\ifmmode\Longrightarrow \else
        \unskip${}\Longrightarrow{}$\ignorespaces\fi}
\def\implies{\ifmmode\Rightarrow \else
        \unskip${}\Rightarrow{}$\ignorespaces\fi}
\def\iff{\ifmmode\Longleftrightarrow \else
        \unskip${}\Longleftrightarrow{}$\ignorespaces\fi}

\let\:=\colon
\newtheorem{Theorem}{Theorem}[section]
 \newtheorem{Lemma}[Theorem]{Lemma}
 \newtheorem{Corollary}[Theorem]{Corollary}
 \newtheorem{Proposition}[Theorem]{Proposition}

 \newtheorem{Definition}[Theorem]{Definition}

\let\epsilon\varepsilon
\let\kappa=\varkappa
\def\qed{\ifhmode\textqed\fi
      \ifmmode\ifinner\quad\qedsymbol\else\dispqed\fi\fi}
\def\textqed{\unskip\nobreak\penalty50
       \hskip2em\hbox{}\nobreak\hfil\qedsymbol
       \parfillskip=0pt \finalhyphendemerits=0}
\def\dispqed{\rlap{\qquad\qedsymbol}}

\opn\dis{dis}
\def\pnt{{\raise0.5mm\hbox{\large\bf.}}}

\opn\Lex{Lex}

\begin{document}

\title{ Frobenius Categories  over  a Triangular Matrix Ring and Comma Categories }

\author{Dancheng Lu}

\address{School  of Mathematical Sciences, Soochow University, 215006 Suzhou, P.R.Chinaa}
\email{ludancheng@suda.edu.cn}

\author{Panpan Xie}

\address{School  of Mathematical Sciences, Soochow University, 215006 Suzhou, P.R.Chinaa}
\email{20184207032@stu.suda.edu.cn}

\keywords{pushout, pullback, resolving subcategory, co-resolving subcategory, Frobenius category, injective, projective, triangulated equivalence, recollement  }

\subjclass[2010]{16D40; 16D50; 16D90}

\begin{abstract} We introduce the dual notions of $\mathcal{E}(\mathcal{X},M,\mathcal{Y})$ and $\mathcal{M}(\mathcal{X},M,\mathcal{Y})$, and  investigate when they have enough injective objects or projective objects, when they are resolving or co-resolving, and when they are Frobenius categories.  In the case that they are Frobenius categories,  we establish a recollement of their stable categories.  Finally some applications to comma categories are given.
\end{abstract}

\maketitle

\section{Introduction}

Throughout this paper, all rings are nonzero associative rings with identity and all modules are unitary and left  unless otherwise stated. For a ring $R$, we write $R$-Mod for the category of all left $R$-modules.
  Given rings $A$, $B$ and an $(A,B)$-bimodule  $M$, we can build an abelian category, denoted by $\mathrm{Rep}(A,M,B)$,  which contains $A$-Mod and $B$-Mod as full subcategories. The objects of $\mathrm{Rep}(A,M,B)$ are all triples $\begin{bmatrix}X\\Y\end{bmatrix}_{\phi}$   such that $X\in A\text{-Mod}$, $Y\in B\text{-Mod}$ and $\phi:M\otimes_B Y\rightarrow X$ is an $A$-homomorphism.  A morphism from $\begin{bmatrix}X\\Y\end{bmatrix}_{\phi}$  to   $\begin{bmatrix}X_1\\Y_1\end{bmatrix}_{\phi_1}$    is a pair $\begin{bmatrix}f\\g\end{bmatrix}$ such that   $f\in \Hom_A(X,X')$, $g\in \Hom_B(Y,Y')$ and the following diagram is commutative:  \begin{equation*}\xymatrix{
                                                                                M\otimes_B Y \ar[d]_{1_M\otimes g} \ar[rr]^(.6){\phi} && X \ar[d]^{f} \\
                                                                                 M\otimes_B Y'  \ar[rr]^(.6){\phi'} && X' .  }
                                                  \end{equation*}
It is well-known that $\mathrm{Rep}(A,M,B)$ is equivalent to the module category $\Lambda$-Mod, where \begin{equation*}
\Lambda:=\begin{bmatrix}A&M\\0&B\end{bmatrix}
\end{equation*} is a formal triangular
matrix ring.

Formal triangular matrix rings play an important
role in ring theory and the representation theory of algebra. They are often used to construct  examples and counterexamples, which make the theory of rings and modules more abundant and interesting. So the properties of formal triangular matrix rings and modules over them have attracted
  many interests (see \cite{AS,AR,GR,HV1,HV2,Mao1,Mao2,XZZ}). These investigations usually focus  to build some special subclasses or subcategories of $\Lambda$-Mod by using corresponding  subclasses or subcategories of $A$-Mod and $B$-Mod.

Projective objects and injective objects of $\mathrm{Rep}(A,M,B)$ are classified by A. Haghany and K. Varadarajan in \cite{HV1} and \cite{HV2} respectively.

 Let $U=\begin{bmatrix}X\\Y\end{bmatrix}_{\phi}\in \mathrm{Rep}(A,M,B)$.
\begin{Theorem} \label{pro}
 {\rm (}\cite[Theorem 3.1]{HV2} {\rm )}  $U$ is a projective object of $\mathrm{Rep}(A,M,B)$ if and only if $Y$ is a projective $B$-module,  $\phi$ is injective and $\Coker(\phi)$ is a projective $A$-module.
\end{Theorem}

Denote by $\widetilde{\phi}$  the morphism $\tau_{X,Y}(\phi)$,  where $\tau_{X,Y}$ is  the natural isomorphism from $\Hom_A(M\otimes_B Y, X)$ to $\Hom_B(Y, \Hom_A(M,X))$.

\begin{Theorem}\label{in}
 {\rm (}\cite[Proposition 5.1]{HV1} {\em and} \cite[p. 956]{AS}{\rm )} $U$ is an injective object of $\mathrm{Rep}(A,M,B)$ if and only if $X$ is an injective $A$-module,  $\widetilde{\phi}$ is surjective and $\Ker(\widetilde{\phi})$ is an injective $B$-module.
\end{Theorem}

 In order to simplify the notation, we introduce the category $\mathrm{Rep}_h(A,M,B)$. Its objects are all triples  $(X,Y)_{\varphi}$ such that $X\in A\text{-Mod}$, $Y\in B\text{-Mod}$  and $\varphi: Y\rightarrow \mathrm{Hom}_A(M,X)$ is a $B$-homomorphism.  A morphism from  $(X,Y)_{\varphi}$ to $(X_1, Y_1)_{\varphi_1}$ is a pair $(f,g)$ such that $f\in \Hom_A(X,X_1)$, $g\in \Hom_B(Y,Y_1)$ and the following diagram is commutative:
\begin{equation*}
\xymatrix{
  Y \ar[d]^{g} \ar[rr]^(.35){\varphi} && \mathrm{Hom}_A(M,X)\ar[d]^{\mathrm{Hom}_A(M,f)}\\
  Y_1\ar[rr]^(.35){\varphi_1} && \mathrm{Hom}_A(M,X_1).                   }
\end{equation*}
It is easy to see that $\mathrm{Rep}_h(A,M,B)$ and  $\mathrm{Rep}(A,M,B)$ are equivalent as abelian categories. Furthermore, using this notation,  Theorem~\ref{in} can be restated as follows:
\begin{Theorem}\label{in1}
 A triple $(X,Y)_{\varphi}$ is an injective object of $\mathrm{Rep}_h(A,M,B)$ if and only if $X$ is an injective $A$-module, $\Ker(\varphi)$ is an injective $B$-module and $\varphi$ is surjective.
 \end{Theorem}

Inspired by  these  results,   we  introduce and study  the following notions.

\begin{Definition}\label{definition1} \em  Let $\mathcal{X}$ (resp.  $\mathcal{Y}$) be a full category of $A$-Mod (resp. $B$-Mod) closed under isomorphisms and containing  zero module. We define $\mathcal{M}(\mathcal{X},M,\mathcal{Y})$ to be the full category of $\mathrm{Rep}(A,M,B)$ whose objects are triples  $\begin{bmatrix}X\\Y\end{bmatrix}_{\phi}\in \mathrm{Rep}(A,M,B)$
such that $Y\in \mathcal{Y}$,  $\Coker(\phi)\in \mathcal{X}$ and $\phi$ is injective.

 Dually, $\mathcal{E}(\mathcal{X},M,\mathcal{Y})$ is defined to be the full category of $\mathrm{Rep}_h(A,M,B)$ whose objects are triples  $(X,Y)_{\varphi}\in \mathrm{Rep}_h(A,M,B)$
such that $X\in \mathcal{X}$, $\Ker(\varphi)\in \mathcal{Y}$ and $\varphi$ is surjective.
\end{Definition}
The notions of $\mathcal{E}(\mathcal{X},M,\mathcal{Y})$  and $\mathcal{M}(\mathcal{X},M,\mathcal{Y})$ are dual to each other and so are their properties.
  We write $\mathcal{M}(A,M,\mathcal{Y})$ and $\mathcal{E}(A,M,\mathcal{Y})$ for $\mathcal{M}(\mathcal{X},M,\mathcal{Y})$ and  $\mathcal{E}(\mathcal{X},M,\mathcal{Y})$ respectively if $\mathcal{X}$ is $A\text{-Mod}$. The categories $\mathcal{M}(A,M,\mathcal{Y})$,  $\mathcal{M}(A,M,B)$ and etc are similarly defined. Recall that   $\mathcal{M}(A,M,B)$ and  $\mathcal{E}(A,M,B)$ have been  studied in \cite{XZZ} in the context of artinian algebras.

    The paper is arranged as follows. In Section 3, we study the properties of $\mathcal{E}(\mathcal{X},M,\mathcal{Y})$. First, we classify  projective objects of   $\mathcal{E}(\mathcal{X},M,\mathcal{Y})$ when it is an exact category and prove that  $\mathcal{E}(\mathcal{X},M,\mathcal{Y})$ has enough projective objects if and only if $\mathcal{X}$ and $\mathcal{Y}$ have enough projective objects under  mild conditions. We discuss when $\mathcal{E}(\mathcal{X},M,\mathcal{Y})$ is  co-resolving and show there is a unique largest co-resolving subcategory of the form $\mathcal{E}(\mathcal{X},M,\mathcal{Y})$.  Based on these results, we characterize when $\mathcal{E}(\mathcal{X},M,\mathcal{Y})$ is a Frobenius category. Finally, it is shown that the stable category of $\mathcal{E}(\mathcal{X},M,\mathcal{Y})$ is a triangulated  recollement of stable categories of $\mathcal{X}$ and $\mathcal{Y}$ when they are all Frobenius categories.  In Section 4 we write  down  the dual results on $\mathcal{M}(\mathcal{X},M,\mathcal{Y})$   without giving their proofs. In the last section,  some of results  in  previous sections are extended to those of some special subcategories of comma categories.

\section{Preliminaries}

In this section we recall some notions and results which we will use later.

We refer to \cite[Section 2.1]{H} for the definition of an {\it  exact category}.  Let  $\mathcal{C}$ be a full subcategory of an abelian category which is closed under extensions. Then $\mathcal{C}$ carries a canonical exact structure. An object $P\in \mathcal{C}$ is  {\it projective} if every short exact sequence $0\rightarrow X_1\rightarrow X_2\rightarrow X_3\rightarrow 0$ in $\mathcal{C}$ induces a short exact sequence  $0\rightarrow \Hom_{\mathcal{C}}(P,X_1)\rightarrow\Hom_{\mathcal{C}}(P, X_2)\rightarrow \Hom_{\mathcal{C}}(P,X_3)\rightarrow 0$. The exact category  $\mathcal{C}$ {\it has enough projective objects} if every $X\in \mathcal{C}$ induces a short exact sequence $0\rightarrow K\rightarrow P\rightarrow X\rightarrow 0$ in $\mathcal{C}$ such that $P$ is projective.   {\it Injective} objects of $\mathcal{C}$ and that $\mathcal{C}$ {\it has enough injective objects} are defined dually. We write $\mathcal{I}(\mathcal{C})$ and $\mathcal{P(C)}$ for the classes of injective objects and projective objects of  $\mathcal{C}$ respectively. An exact category is a {\it Frobenius category} if it has enough projective objects and injective objects, and if its projective and injective objects coincide, i.e., $\mathcal{I(C)}=\mathcal{P(C)}$.

The notion of a Frobenius category is very important because its stable category is a triangulated category and many triangulated categories including derived category and  homotopy category of an abelian category are {\it algebraic}. Recall that a triangulated category is  algebraic if it is the stable category of a Frobenius category.  Let $\mathcal{C}$ be a Frobenius category. By definition, the stable category $\mathrm{St}\ \mathcal{C}$ has the same objects as $\mathcal{C}$ and for any $X,Y\in \mathcal{C}$, $$\Hom_{\mathrm{St}\ \mathcal{C}}(X,Y):=\Hom_{\mathcal{C}}(X,Y)/\mathcal{P}(X,Y).$$
Here, $\mathcal{P}(X,Y)$ stands for the set of morphisms from $X$ to $Y$ that factor through a projective object. For each $X\in B$ we fix a short exact sequence $0\rightarrow X \rightarrow I(X)\rightarrow T X\rightarrow 0$ such that $I(X)$ is an injective object of $\mathcal{C}$.     Let $\xi: 0\rightarrow X\xrightarrow{f}Y\xrightarrow{g}Z\rightarrow 0$ be a short exact sequence in $\mathcal{C}$. Then it induces a commutative diagram:
\begin{equation}\tag{Fig 1}\label{CD} \xymatrix{0\ar[r] &X\ar[r]^{f}\ar@{=}[d]&Y\ar[r]^g\ar[d]&Z\ar[d]^{h}\ar[r]&0\\
0\ar[r] &X\ar[r]&I(X)\ar[r]&TX\ar[r]&0.}
\end{equation}
The sequence $(f,g,h)$ is called a {\it standard triangle} induced by $\xi$.  The image in $\mathrm{St}\ \mathcal{C}$  of such a sequence is an {\it exact triangle} of $\mathrm{St}\ \mathcal{C}$ and every exact triangle of $\mathrm{St}\ \mathcal{C}$ arises in this  way, see \cite[Proposition 3.3.2]{H}.

\begin{Lemma} \label{tri}  Let $\mathcal{C}$ and $\mathcal{D}$ be Frobenius categories. If  $\mathbf{F}:\mathcal{C}\rightarrow \mathcal{D}$ is an additive functor satisfying
$\mathrm{(1)}$ $\mathbf{F}$ is exact, $\mathrm{(2)}$ $\mathbf{F}(\mathcal{I}(\mathcal{C}))\subseteq \mathcal{I}(\mathcal{D})$, then $\overline{\mathbf{F}}: \mathrm{St}\  \mathcal{C}\rightarrow \mathrm{St}\ \mathcal{D}$ is a triangulated functor.
\end{Lemma}

\begin{proof} Note that $\mathbf{F}$ turns \ref{CD} in $\mathcal{C}$ into the following  commutative diagram in $\mathcal{D}$
\begin{equation*}\xymatrix{0\ar[r] &F(X)\ar[r]^{F(f)}\ar@{=}[d]&F(Y)\ar[r]^{F(g)}\ar[d]&F(Z)\ar[d]^{F(h)}\ar[r]&0\\
0\ar[r] &F(X)\ar[r]&F(I(X))\ar[r]&F(TX)\ar[r]&0,}
\end{equation*} whose rows are exact. This implies $\overline{\mathbf{F}}$ maps an {\it exact triangle} in $\mathrm{St}\ \mathcal{C}$ into an exact triangle in $\mathrm{St}\ \mathcal{D}$.
\end{proof}

The tools of pushout and pullback are indispensable.  We record the following  results for the later use.

\begin{Lemma} {\em(see e.g. \cite[Proposition 5.2]{I})}\label{pushout1}\
Let \begin{equation*} \label{outback}
  \xymatrix{
  & A   \ar[d]_a \ar[r]^f
                 & B \ar[d]^b       \\
  & C \ar[r]^g   & D
   }
  \end{equation*}
  be a commutative diagram in an abelian category.

   $\mathrm{(1)}$   If it is a pullback  diagram, then

$\mathrm{(i)}$ $\mathrm{Ker} a\stackrel{\widetilde{f}}{\cong} \mathrm{Ker} b$, $\mathrm{Ker} f\stackrel{\widetilde{a}}{\cong} \mathrm{Ker} g$;

$\mathrm{(ii)}$ both  $\widetilde{g}: \mathrm{Coker} a\rightarrow \mathrm{Coker} b$ and $\widetilde{b}: \mathrm{Coker} f\rightarrow \mathrm{Coker} g$ are injective;

$\mathrm{(iii)}$ $(b,g)$ is the pushout of  $(a,f)$ if and only if either $\widetilde{g}$ or $\widetilde{b}$  is an isomorphism.
Dually,

$\mathrm{(2)}$ If it is a pushout  diagram, then

$\mathrm{(i)}$ $\mathrm{Coker} a\stackrel{\widetilde{g}}{\cong} \mathrm{Coker} b$, $\mathrm{Coker} f\stackrel{\widetilde{b}}{\cong} \mathrm{Coker} g$;

$\mathrm{(ii)}$ both  $\widetilde{f}: \mathrm{Ker} a\rightarrow \mathrm{Ker} b$ and $\widetilde{a}: \mathrm{Ker} f\rightarrow \mathrm{Ker} g$ are surjective.

$\mathrm{(iii)}$ $(a,f)$  is the pullback of $(b,g)$  if and only if either $\widetilde{f}$ or $\widetilde{a}$  is an isomorphism.

\end{Lemma}

\begin{Lemma} {\em (see e.g. \cite[Proposition 2.12]{Bu})} \label{pushout2}   If we have one of the following commutative diagram in an abelian category with exact rows:
\begin{equation*}
  \xymatrix{
 0\ar[r]& E\ar@{=}[d]\ar[r] & A   \ar[d]_a \ar[r]^f
             & B \ar[d]^b  \ar[r]&0     \\
 0\ar[r]  & E\ar[r]   & C \ar[r]^g   & D \ar[r]&0 ,
   }
  \end{equation*}

and  \begin{equation*}
  \xymatrix{
 0\ar[r] & A   \ar[d]_a \ar[r]^f
             & B \ar[d]^b\ar[r] & E\ar@{=}[d]\ar[r] \ar[r]&0     \\
 0\ar[r]     & C \ar[r]^g   & D\ar[r] & E\ar[r]\ar[r]&0 ,
   }
  \end{equation*} then the following \begin{equation*}
  \xymatrix{
  & A   \ar[d]_a \ar[r]^f
                 & B \ar[d]^b       \\
  & C \ar[r]^g   & D
   }
  \end{equation*} is a pushout-pullback diagram. (This means $(a,f)$ is a pullback $(g,b)$, and  $(g,b)$ is a pushout $(a,f)$ at the same time)
\end{Lemma}

\section{The properties of  $\mathcal{E}(\mathcal{X},M,\mathcal{Y})$ \label{3}}

In this section, $A$ and $B$ are rings and $M$ is an $(A,B)$-bimodule. Let $\mathcal{X}$ and $\mathcal{Y}$ be subcategories of $A$-Mod and $B$-Mod respectively closed under isomorphisms and containing the zero module. We will study properties of  $\mathcal{E}(\mathcal{X},M,\mathcal{Y})$ (see Definition~\ref{definition1} for what it means). First of all, we need to define some useful functors.

(1) $\mathbf{p}: \mathcal{X}\rightarrow \mathcal{E}(\mathcal{X},M,\mathcal{Y})$,\qquad   $X\longmapsto (X,\Hom_A(M,X))_1$;
\vspace{1mm}

(2) $\mathbf{q}:\mathcal{Y}\rightarrow \mathcal{E}(A,M,\mathcal{Y})$, \qquad $Y\longmapsto (0,Y)_0$;
\vspace{1mm}

(3) $\mathbf{P}:  \mathcal{E}(\mathcal{X},M,\mathcal{Y})\rightarrow \mathcal{X}$, \qquad $(X,Y)_{\varphi}\longmapsto X$;

\vspace{1mm}

(4) $\mathbb{Q}:  \mathcal{E}(\mathcal{X},M,\mathcal{Y})\rightarrow \mathcal{Y}$,\qquad $(X,Y)_{\varphi}\longmapsto \Ker(\varphi)$;

(5) $\mathbf{Q}:  \mathcal{E}(\mathcal{X},M,\mathcal{Y})\rightarrow B\mathrm{\text{-}Mod}$,\qquad $(X,Y)_{\varphi}\longmapsto Y$.
\vspace{1mm}

The actions to morphisms   of these functors  are defined in a  natural way. For examples, if $f\in \mathrm{Mor} \mathcal{X}$, then $\mathbf{p}(f)=(f, \Hom(M,f))$; if $(f,g)\in \mathrm {Mor} \mathcal{E}(\mathcal{X},M,\mathcal{Y})$, then $\mathbf{P}(f,g)=f$, $\mathbf{Q}(f,g)=g$
and  $\mathbb{Q}(f,g)=\overline{g}$, where $\overline{g}$ is given by the following commutative diagram:
\begin{equation*} \xymatrix{
0\ar[r]&\Ker(\varphi)\ar[d]^{\overline{g}}\ar[r]&  Y \ar[d]^{g} \ar[rr]^(.35){\varphi} && \mathrm{Hom}_A(M,X)\ar[d]^{\mathrm{Hom}_A(M,f)}\ar[r]&0\\
0\ar[r]&\Ker(\varphi_1)\ar[r] & Y_1\ar[rr]^(.35){\varphi_1} && \mathrm{Hom}_A(M,X_1)\ar[r]&0.  }
\end{equation*}

\vspace{2mm}
Note that $\mathbf{Q} (X,Y)_{\varphi}$ may not belong to $\mathcal{Y}$ even if $(X,Y)_{\varphi}\in \mathcal{E}(\mathcal{X},M,\mathcal{Y})$. If the categories involved  are exact, then the functors $\mathbf{q}$, $\mathbb{Q}$, $\mathbf{P}$ and $\mathbf{Q}$ are exact. On the other hand,  $\mathbf{p}$ is faithful and left exact, but not always exact.

\begin{Proposition} \label{extension}
 If both $\mathcal{X}$ and $\mathcal{Y}$ are closed under extensions, then $\mathcal{E}(\mathcal{X},M,\mathcal{Y})$ is closed under extensions;
\end{Proposition}

\begin{proof}   Let $$0\rightarrow (X_1,Y_1)_{\varphi_1}\xrightarrow{(f_1,g_1)}(X_2,Y_2)_{\varphi_2}\xrightarrow{(f_2,g_2)}(X_3,Y_3)_{\varphi_3}\rightarrow 0$$ be a short exact sequence in $\mathrm{Rep}_h(A,M,B)$ such that  $(X_i,Y_i)_{\varphi_i}\in \mathcal{E}(\mathcal{X},M,\mathcal{Y})$ for $i=1,2$. Since the sequence $0\rightarrow X_1\xrightarrow{f_1}X_2\xrightarrow{f_2}X_3\rightarrow 0$ is exact, it follows that $X_2\in \mathcal{X}$.
In addition,  we have the following commutative diagram:
\begin{equation*}\xymatrix{
      0\ar[r]   & Y_1 \ar[rr]^{g_1} \ar[d]^{\varphi_1} & &Y_2 \ar[d]_{\varphi_2} \ar[rr]^{g_2} && Y_3 \ar[d]^{\varphi_3}\ar[r]&0 \\
        0\ar[r]&  \Hom_A(M,X_1) \ar[rr]^{\Hom_A(M,f_1)}& &\Hom_A(M,X_2)\ar[rr]^{\Hom_A(M,f_2)} &&\Hom_A(M,X_3) ,}
\end{equation*}
 where both rows are exact.  By the Snake Lemma  we obtain the exact sequence: $$0\rightarrow \Ker(\varphi_1)\rightarrow\Ker(\varphi_2)\rightarrow\Ker(\varphi_3)\rightarrow\Coker(\varphi_1)\rightarrow\Coker(\varphi_2)\rightarrow\Coker(\varphi_3). $$ Note that $\Coker(\varphi_1)=\Coker(\varphi_3)=0$ and $\Ker(\varphi_i)\in \mathcal{X}$ for $i=1,2$, it follows that $\Coker(\varphi_2)=0$ and $\Ker(\varphi_2)\in \mathcal{X}.$
\end{proof}

   If $\mathcal{E}(\mathcal{X},M,\mathcal{Y})$ is closed under extensions, then so is $\mathcal{Y}$, since $\mathbf{q}$ is an exact functor; but  $\mathcal{X}$ may be not closed under extensions in general.   If $\mathbf{p}$ is  an exact functor, then  $\mathcal{X}$ is also closed under extensions.
\vspace{3mm}

Next, we aim to classify  all projective objects of  $\mathcal{E}(\mathcal{X},M,\mathcal{Y})$ when it is an exact category.  We begin with a lemma,  which  deals with this question under the hypothesis as weak as possible.
 To this end we introduce  condition (*).

(*)  {\it  Any short exact sequence $0\rightarrow X_1\xrightarrow{f_1} X_2\xrightarrow{f_2} X_3\rightarrow 0$ in $\mathcal{X}$ induces a short exact sequence $0\rightarrow (X_1,Y_1)_{\varphi_1}\xrightarrow{(f_1,g_1)} (X_2,Y_2)_{\varphi_2}\xrightarrow{(f_2,g_2)} (X_3,Y_3)_{\varphi_3}\rightarrow 0$ in $\mathcal{E}(\mathcal{X},M,\mathcal{Y}).$}

 In the case when $\mathbf{p}$ is exact, this condition holds automatically.

\begin{Lemma} \label{projective}  Suppose that both $\mathcal{X}$ and $\mathcal{Y}$ are closed under extensions and let $(P,Q)_{\varphi}\in \mathcal{E}(\mathcal{X},M,\mathcal{Y})$.  If the following conditions holds:

$\mathrm{(i)}$ $P$ is a  projective object in $\mathcal{X}$;

$\mathrm{(ii)}$ $Q$ is $\mathcal{Y}$-projective, (this means $\Hom_A(Q, )$ takes a short exact sequence in $\mathcal{Y}$ to a short exact sequence in $B$-Mod,  but $Y$ may not belong to $\mathcal{Y}$);

$\mathrm{(iii)}$ the map $\varphi$ satisfies the following property: given  any short exact sequence in $B$-Mod: $$0\rightarrow \Ker(\phi)\rightarrow Y\xrightarrow{\phi} \Hom_A(M,X)\rightarrow 0$$ with $\Ker(\phi)\in \mathcal{Y}$ and $X\in \mathcal{X}$,  and any $A$-homomorphism $f:P\rightarrow X$, there exists a $B$-homomorphism: $g:Q\rightarrow Y$ such that $\Hom(M,f)\circ\varphi=\phi\circ g$, that is, the following diagram is commutative:
\begin{equation*}\xymatrix{
&&&Q\ar[d]^{\varphi}\ar@{-->}[ddl]_g\\
&&&\Hom_A(M,P)\ar[d]^{\Hom(M,f)}\\
0\ar[r]& \Ker(\phi)\ar[r]& Y\ar[r]^(.35){\phi}&\Hom_A(M,X)\ar[r]&0,}
\end{equation*}
then  $(P,Q)_{\varphi}$ is a projective object of $\mathcal{E}(\mathcal{X},M,\mathcal{Y})$.

 The converse is also true if  condition (*) holds.
\end{Lemma}

\begin{proof} First we assume that $(P,Q)_{\varphi}$ satisfies the  three conditions mentioned above, and we proceed to prove $(P,Q)_{\varphi}$ is a projective object.  For this, we let  $$0\rightarrow (X_1,Y_1)_{\varphi_1}\xrightarrow{(f_1,g_1)}(X_2,Y_2)_{\varphi_2}\xrightarrow{(f_2,g_2)}(X_3,Y_3)_{\varphi_3}\rightarrow 0$$
be a short exact sequence in $\mathcal{E}(\mathcal{X},M,\mathcal{Y})$ and $(\alpha,\beta): (P,Q)_{\varphi}\rightarrow (X_3,Y_3)_{\varphi_3}$ a morphism in $\mathrm{Rep}_h(A,M,B)$.  Since $P$ is a projective object of $\mathcal{X}$, there is an $A$-morphism $\rho: P\rightarrow X_2$ such that $\alpha=f_2\rho$.  Denote $\Hom(M,h)$ by $h^*$ for any $A$-homomorphism $h$.  By (iii), there exists a $B$-homomorphism $t$ such that $\rho^*\varphi=\varphi_2t$.  Considering  the following commutative diagram, where all rows and columns are exact.

\begin{equation*}
\xymatrix{& 0 \ar[d]&0\ar[d]&&0\ar[d]\\
0\ar[r]& \Ker(\varphi_1) \ar[dd]\ar[r] & \Ker(\varphi_2)  \ar[dd]^{i_2} \ar[rr]^{\overline{g_2}}&& \Ker(\varphi_3) \ar[dd]^{i_3}\ar[r] &0 \\
 & &&Q\ar[dr]^{\beta}\ar@{-->}[dl]^t\ar[dd]^(.65){\varphi}\ar@{-->}[ur]^{j}\ar@{-->}[ul]^{k} &\\
0\ar[r] & Y_1\ar[dd]^{\varphi_1}\ar[r]^{g_1}&Y_2\ar[dd]^(.5){\varphi_2}\ar[rr]^(.37){g_2}&&Y_3\ar[dd]^{\varphi_3}\ar[r]&0\\
 & && \Hom_A(M,P)\ar[dr]^{\alpha^*}\ar@{-->}[dl]^{\rho^*}&\\
   0\ar[r]&  \Hom_A(M,X_1) \ar[d]\ar[r]_{f_1^*} &\Hom_A(M, X_2)\ar[rr]_{f_2^*}\ar[d]&& \Hom_A(M,X_3)\ar[d]\\
   &0&0&&0.}
\end{equation*}
\noindent The exactness of the upmost row follows from the Snake Lemma.
 Note that $$\varphi_3\beta=\alpha^*\varphi=f_2^*\rho^*\varphi=f_2^*\varphi_2t=\varphi_3g_2t,$$ it follows that $\varphi_3(\beta-g_2t)=0,$ and so there exists a $B$-homomorphism $j:Q\rightarrow \Ker(\varphi)$ such that $\beta-g_2t=i_3j$ by the universal property of the Kernel. By (ii), we have $j=g_2k$ for some $B$-homomorphism $k:Q\rightarrow \Ker(\varphi_2)$.

Now, set $s=t+i_2k$. Then $$\varphi_2s=\varphi_2t+\varphi_2i_2k=\varphi_2t=\rho^*\varphi$$
\qquad \qquad \qquad\qquad and \begin{equation*}\begin{aligned}g_2s&=g_2t+g_2i_2k=g_2t+i_3\overline{g_2}k\\
&=g_2t+i_3j=g_2t+\beta-g_2t\\
&=\beta.
\end{aligned}
\end{equation*}
Hence $(\rho, s): (P,Q)_{\varphi}\rightarrow (X_2,Y_2)_{\varphi_2}$ is a morphism in $\mathrm{Rep}_h(A,M,B)$ and it lifts the morphism $(\alpha,\beta)$, that is, $(\alpha,\beta)=(f_2,g_2)(\rho, s)$. Consequently, $(P,Q)_{\varphi}$ is a projective object of  $\mathcal{E}(\mathcal{X},M,\mathcal{Y})$, as desired.

Conversely,  assume that  $(P,Q)_{\varphi}$ is a projective object in $\mathcal{E}(\mathcal{X},M,\mathcal{Y})$. First, we prove that (iii) holds.  Let $$0\rightarrow \Ker(\phi)\rightarrow Y\xrightarrow{\phi} \Hom_A(M,X)\rightarrow 0$$ be a short exact sequence  in $B$-Mod with $\Ker(\phi)\in \mathcal{Y}$ and $X\in \mathcal{X}$,  and $f:P\rightarrow X$ an  $A$-homomorphism. Then we have the following exact sequences in $\mathcal{E}(\mathcal{X},M,\mathcal{Y})$:$$0\rightarrow (0,\Ker(\phi))_0\rightarrow (X,Y)_{\phi}\xrightarrow{(1,\phi)}(X, \Hom_A(M,X))_1\rightarrow 0.$$
  For the morphism $(f, \Hom_A(M,f)\varphi): (P,Q)_{\varphi}\rightarrow (X,\Hom_A(M,X))_1$,  there exists a morphism, say   $(j,k):(P,Q)_{\varphi}\rightarrow (Y,\Hom_A(M,X))_{\phi}$ such that $(f, \Hom(M,f)\varphi)=(1,\phi)(j,g)$. Particularly, we have $\phi g=\Hom_A(M,f)\varphi$. This proves (iii).

 Next, since any short sequence in $B$-Mod: $0\rightarrow Y_1\rightarrow Y_2\rightarrow Y_3\rightarrow 0$ induces a short sequence $0\rightarrow (0,Y_1)_0\rightarrow (0,Y_2)_0\rightarrow (0,Y_3)_0\rightarrow 0$,    condition  (ii) follows.

 Finally, we prove (i) under the additional condition (*). Let  $0\rightarrow X_1\xrightarrow{f_1} X_2\xrightarrow{f_2} X_3\rightarrow 0$  be a short exact sequence in $\mathcal{X}$ and let $\alpha: P\rightarrow X_3$ be a $A$-homomorphism. Then,  by  condition (*), there exists a short exact sequence of the form: $$0\rightarrow (X_1,Y_1)_{\varphi_1}\xrightarrow{(f_1,g_1)} (X_2,Y_2)_{\varphi_2}\xrightarrow{(f_2,g_2)} (X_3,Y_3)_{\varphi_3}\rightarrow 0$$ in $\mathcal{E}(\mathcal{X},M,\mathcal{Y}).$ Using (iii), which have been proved before, there exists a morphism $(\alpha,\beta):(P,Q)_{\varphi}\rightarrow (X_3,Y_3)_{\varphi_3}$ in $\mathcal{E}(\mathcal{X},M,\mathcal{Y})$. It follows that  there exists a morphism $(\alpha',\beta'):(P,Q)_{\varphi}\rightarrow (X_2,Y_2)_{\varphi_2}$ such that $(\alpha,\beta)=(\alpha',\beta')(f_2,g_2)$ and this particularly implies $\alpha=\alpha'f_2$. Hence $P$ is a projective object of $\mathcal{X}$,  proving (i).
\end{proof}

Hereafter we need the additional hypothesises that $\Hom_A(M,X)\in \mathcal{Y}$ for $X\in \mathcal{X}$ and $\mathrm{Ext}^1_A(M,X)=0$  for any $X\in \mathcal{X}$. For short, we will  denote them by $\Hom_A(M,\mathcal{X})\subseteq \mathcal{Y}$ and $\mathrm{Ext}^1_A(M,\mathcal{X})=0$  respectively.
  Lemma~\ref{projective} can be simplified a lot under these hypothesises.

\begin{Corollary} \label{projective2} Suppose that $\mathrm{Ext}^1_A(M,\mathcal{X})=0$ and that $\mathcal{E}(\mathcal{X},M,\mathcal{Y})$ is closed under extensions. If either $\mathcal{Y}$ has enough injective objects or $\Hom_A(M,\mathcal{X})\subseteq \mathcal{Y}$, then for any $(P,Q)_{\varphi}\in \mathcal{E}(\mathcal{X},M,\mathcal{Y})$,  $(P,Q)_{\varphi}$ is a projective object of $\mathcal{E}(\mathcal{X},M,\mathcal{Y})$ if and only if
$P$ is a projective object of $\mathcal{X}$ and $Q$ is  $\mathcal{Y}$-projective.
\end{Corollary}
\begin{proof} Since $\mathrm{Ext}^1_A(M,\mathcal{X})=0$,  the functor $\mathbf{p}$ is exact, and particularly,   condition (*) holds. Thus,  if $(P,Q)_{\varphi}$ is a projective object of $\mathcal{E}(\mathcal{X},M,\mathcal{Y})$, then
$P$ is a projective object of $\mathcal{X}$ and $Q$ is  $\mathcal{Y}$-projective by Lemma~\ref{projective}. Conversely, we assume that  $P$ is a projective object of $\mathcal{X}$ and $Q$ is  $\mathcal{Y}$-projective. If $\mathcal{Y}$ has enough injective objects, then  $\Ext^1_B(Q, Y)=0$ for $Y\in \mathcal{Y}$ by the Long Exact Sequence Theorem, and so  condition (iii) in Lemma~\ref{projective} is satisfied (by noting that  $\Ext^1_B(Q, \Ker(\phi))=0$); If $\Hom_A(M,\mathcal{X})\subseteq \mathcal{Y}$, then the short exact sequence $0\rightarrow \Ker(\phi)\rightarrow Y\xrightarrow{\phi} \Hom_A(M,X)\rightarrow 0$ in  condition (iii) lies in $\mathcal{Y}$ and thus  condition (iii) is also satisfied.  Consequently, $(P,Q)_{\varphi}$ is a projective object of $\mathcal{E}(\mathcal{X},M,\mathcal{Y})$ by Lemma~\ref{projective}.
\end{proof}

We now consider when $\mathcal{E}(\mathcal{X},M,\mathcal{Y})$ has enough projective objects.

\begin{Theorem} \label{enough}  Suppose that $\mathcal{E}(\mathcal{X},M,\mathcal{Y})$ is closed under extensions and assume further that $\mathrm{Ext}^1_A(M,\mathcal{X})=0$  and   $\Hom_A(M,\mathcal{X})\subseteq \mathcal{Y}$. Then $\mathcal{E}(\mathcal{X},M,\mathcal{Y})$ has enough projective objects if and only if both $\mathcal{X}$ and $\mathcal{Y}$ have enough projective objects.
\end{Theorem}
\begin{proof} The ``only if" part follows since $\mathbf{P}$ and $\mathbf{Q}$ are exact functors.  Conversely, assume that both $\mathcal{X}$ and $\mathcal{Y}$ have enough projective objects. Let $(X,Y)_f\in \mathcal{E}(\mathcal{X},M,\mathcal{Y}).$
Then there exists a short exact sequence in $\mathcal{X}$: $$0\rightarrow K\rightarrow P\xrightarrow{\alpha}X\rightarrow 0,$$ where $P$ is a projective object of $\mathcal{X}$.  From this, we obtain the following commutative diagram:
\begin{equation*}\xymatrix{&&0\ar[d]&&0\ar[d]\\
&&\Ker(\beta)\ar[d]\ar[rr]^{\cong}_{\overline{g}}&&\Hom_A(M,K)\ar[d]\\
0\ar[r]& \Ker(g)\ar[d]^{\overline{\beta}}_{\cong}\ar[r] &T\ar[rr]^{g}\ar[d]^{\beta}&&\Hom_A(M,P)\ar[d]^{\Hom_A(M,\alpha)}\ar[r]&0\\
0\ar[r]&\Ker(f)\ar[r]&Y\ar[d]\ar[rr]^{f}&&\Hom_A(M,X)\ar[d]\ar[r]& 0\\
&&0&&0,}
\end{equation*}
where $(\beta,g)$ is the pullback of $(f,\Hom(M,\alpha))$. By Lemma~\ref{pushout1}, we see that $\Ker(g)\cong\Ker(f)$, $\Ker(\beta)\cong \Hom_A(M,K)$,  and that both rows and both columns are short exact sequences. Since $Y\in \mathcal{Y}$ and $\Ker(\beta)\cong \Hom_A(M,K)\in \mathcal{Y}$,  it follows that $T\in \mathcal{Y}$ and so we have the following exact sequence in $\mathcal{Y}$: $$0\rightarrow N\rightarrow Q\xrightarrow{e}T\rightarrow 0,$$ where $Q$ is a projective object of $\mathcal{Y}$.
Considering the following commutative diagram:
\begin{equation*}\xymatrix{
&\Ker(\overline{e})\ar@{^(->}[d]\ar[r]^{\cong}&N\ar@{^(->}[d]\\
0\ar[r]&\Ker(\beta e)\ar[r]\ar@{->>}[d]^{\overline{e}}&Q\ar[r]^{\beta e}\ar@{->>}[d]^{e} &Y\ar@{=}[d]\ar[r]&0\\
0\ar[r]&\Ker(\beta)\ar[r]&T\ar[r]^{\beta}&Y\ar[r]&0.}
\end{equation*}
In view of Lemma~\ref{pushout2}, the down-left square is a pushout-pullback diagram. It follows that the left column is also a short exact sequence and so $\Ker(\beta e)\in \mathcal{Y}$. Similarly, we get the following commutative diagram

\begin{equation*}\xymatrix{
&\Ker(e')\ar@{^(->}[d]\ar[r]^{\cong}&N\ar@{^(->}[d]\\
0\ar[r]&\Ker(g e)\ar[r]\ar@{->>}[d]^{e'}&Q\ar[rr]^{g e}\ar@{->>}[d]^{e} &&\Hom_A(M,P)\ar@{=}[d]\ar[r]&0\\
0\ar[r]&\Ker(g)\ar[r]&T\ar[rr]^{g}&&\Hom_A(M,P)\ar[r]&0}.
\end{equation*}
and so $\Ker(ge)\in \mathcal{Y}$. Finally, putting  these facts together, we obtain the following commutative diagram:
\begin{equation*}\xymatrix{&&0\ar[d]&&0\ar[d]\\
0\ar[r]&N\ar[d]\ar[r]&\Ker(\beta e)\ar[d]\ar@{->>}[rr]^{\overline{ge}}&&\Hom_A(M,K)\ar[d]\\
0\ar[r]& \Ker(ge)\ar[d]^{\overline{\beta}e'}\ar[r] &Q\ar[rr]^{ge}\ar[d]^{\beta e}&&\Hom_A(M,P)\ar[d]^{\Hom_A(M,\alpha)}\ar[r]&0\\
0\ar[r]&\Ker(f)\ar[r]&Y\ar[d]\ar[rr]^{f}&&\Hom_A(M,X)\ar[d]\ar[r]& 0\\
&&0&&0,}
 \end{equation*}  and it induces the following short exact sequence in $\mathcal{E}(\mathcal{X},M,\mathcal{Y})$:
 $$0\rightarrow (K,\Ker(\beta e))_{\overline{ge}}\rightarrow (P,Q)_{ge}\xrightarrow{(\alpha,\beta e)} (X,Y)_f\rightarrow 0.$$
 By Corollary~\ref{projective2}, we have $(P,Q)_{ge}$ is  a projective object of $\mathcal{E}(\mathcal{X},M,\mathcal{Y})$ and thus our proof is complete.
\end{proof}

 Let $\mathcal{C}$ be an abelian category which has enough injective objects. A full subcategory $\mathcal{D}$ of $\mathcal{C}$ is a {\it co-resolving} subcategory of $\mathcal{C}$ or just co-resolving if $\mathcal{I(C)}\subseteq \mathcal{D}$ and $\mathcal{D}$ is closed under direct summands, extensions and the cokernels of injective homomorphisms. We turn to consider when $\mathcal{E}(\mathcal{X},M,\mathcal{Y})$ is a co-resolving subcategory of $\mathrm{Rep}_h(A,M,B)$.

\begin{Theorem} \label{co-resolving} The following statements are equivalent:

$\mathrm{(1)}$ $\mathcal{E}(\mathcal{X},M,\mathcal{Y})$ is a co-resolving subcategory of $\mathrm{Rep}_h(A,M,B)$;

$\mathrm{(2)}$ Both $\mathcal{X}$ and $\mathcal{Y}$ are co-resolving and $\mathrm{Ext}^1_A(M, \mathcal{X})=0$.
\end{Theorem}
\begin{proof} Assume that $\mathcal{E}(\mathcal{X},M,\mathcal{Y})$ is a co-resolving subcategory of $\mathrm{Rep}_h(A,M,B)$.  Let $X\in \mathcal{X}$. Then we have a short exact sequence $$0 \rightarrow X\xrightarrow{\alpha} E\rightarrow N\rightarrow 0,$$
where $E$ is an  injective $A\text{-}$module.  This gives rise to an injective homomorphism $$(\alpha, \Hom(M,\alpha)): (X,\Hom_A(M,X))_1\rightarrow (E,\Hom_A(M,E))_1.$$
Write $\mathrm{Coker}(\alpha, \Hom(M,\alpha))$ as $(N',Y)_f.$  Then $(N',Y)_f\in \mathcal{E}(\mathcal{X},M,\mathcal{Y})$ and the following sequence
$$ 0\rightarrow (X,\Hom_A(M,X))_1\xrightarrow{(\alpha, \Hom_A(M,\alpha))} (E,\Hom_A(M,E))_1 \rightarrow (N',Y)_f\rightarrow 0$$ is exact in $\mathrm{Rep}_h(A,M,B)$. From this it follows that $N'\cong N$ and $$0\rightarrow \Hom_A(M,X)\xrightarrow{\Hom_A(M,\alpha)}\Hom_A(M,E) \rightarrow Y\rightarrow 0$$ is a short exact sequence in $B$-Mod.  In addition,
we have the following commutative diagram:
        \begin{equation*} \xymatrix{
      0\ar[r]   & \Hom_A(M,X) \ar[rr]^{\Hom_A(M,\alpha)} \ar[d]^{1} & & \Hom_A(M,E) \ar[d]_{1} \ar[r] & Y \ar[d]^{f}\ar[r]&0 \\
        0\ar[r]&  \Hom_A(M,X) \ar[rr]^{\Hom_A(M,\alpha)} &&\Hom_A(M,E)\ar[r]&\Hom_A(M,N').  }
        \end{equation*}
        In view of the Snake Lemma, it follows that $\Ker(f)=0$  and so $f$ is an isomorphism. Hence  the second row can be completed into  a short exact sequence in $B$-Mod and this implies that $\mathrm{Ext}^1_A(M, X)=0$ by the Long Exact Sequence Theorem.

        Since $\mathrm{Ext}^1_A(M, \mathcal{X})=0$,  $\mathbf{p}$ is an exact functor and so  both $\mathcal{X}$ and $\mathcal{Y}$ is closed under extensions. It is clear that both $\mathcal{X}$ and $\mathcal{Y}$ is closed under direct summands. Finally, the exactness of $\mathbf{p}$ and $\mathbf{q}$  implies both $\mathcal{X}$ and $\mathcal{Y}$ is closed under the cokernels of injective homomorphisms and Theorem~\ref{in1} implies that $\mathcal{X}$ and $\mathcal{Y}$ contains all the injective $A$-modules and injective $B$-modules respectively.

        Conversely, assume that both $\mathcal{X}$ and $\mathcal{Y}$ are co-resolving and $\mathrm{Ext}^1_A(M, \mathcal{X})=0$.  It is clear that $\mathcal{E}(\mathcal{X},M,\mathcal{Y})$  contains all injective objects of $\mathrm{Rep}_h(A,M,B)$ by Theorem~\ref{in1} and is closed under direct summands and extensions. Let $$0\rightarrow  (X,Y)_f\xrightarrow{(\alpha,\beta)} (X',Y')_{f'}\rightarrow (U,V)_g\rightarrow 0$$ be a short exact sequence in $\mathrm{Rep}_h(A,M,B)$ such that  its first and second nonzero terms belong to $\mathcal{E}(\mathcal{X},M,\mathcal{Y})$. Then $U\in \mathcal{X}$. In addition, we have the following commutative diagram:
           \begin{equation*} \label{CD1}\xymatrix{
      0\ar[r]   & Y \ar[rr]^{\beta} \ar[d]^{f} & & Y' \ar[d]_{f'} \ar[r] & V \ar[d]^{g}\ar[r]&0 \\
        0\ar[r]&  \Hom_A(M,X) \ar[rr]^{\Hom_A(M,\alpha)} &&\Hom_A(M,X')\ar[r]&\Hom_A(M,U)\ar[r]&0.  }
        \end{equation*}
        Note that both rows are short exact sequences. In view of the Snake Lemma, we have   $g$ is surjective, and   the following sequence $$0\rightarrow \Ker(f)\rightarrow \Ker(f')\rightarrow \Ker(g)\rightarrow 0$$ is exact. Hence $\Ker(g)\in \mathcal{Y}$ and this implies $ (U,V)_g\in \mathcal{E}(\mathcal{X},M,\mathcal{Y})$, completing the proof.
\end{proof}

In view of Theorem~\ref{co-resolving}, we call $\mathcal{X}$  to be {\it $M^{\bot}$-co-resolving } provided that $\mathcal{X}$ is a co-resolving subcategory of $A$-Mod  such that $\mathrm{Ext}^1_A(M, X)=0$ for any $X\in\mathcal{X}$. The following result shows that there are  a unique largest  $M^{\bot}$- co-resolving subcategory in $A$-Mod and a unique largest   co-resolving subcategory in $\mathrm{Rep}_h(A,M,B)$ of the form $\mathcal{E}(\mathcal{X},M,\mathcal{Y})$.

\begin{Proposition} Set $\mathbb{X}:=\{X\in A\text{-}\mathrm{Mod}|\Ext^i_A(M,X)=0,  \forall i\geq 1\}$.

$\mathrm{(1)}$  $\mathbb{X}$ is  an $M^{\bot}$-co-resolving subcategory.

$\mathrm{(2)}$  Any $M^{\bot}$-co-resolving subcategory of $A$-{\rm Mod}  is included in $\mathbb{X}$.

$\mathrm{(3)}$ $\mathcal{E}(\mathbb{X},M,B)$ is the largest co-resolving subcategory of $\mathrm{Rep}_h(A,M,B)$ of the form $\mathcal{E}(\mathcal{X},M,\mathcal{Y})$.

$\mathrm{(4)}$ $\mathcal{E}(A,M,B)$ is a co-resolving  subcategory of $\mathrm{Rep}_h(A,M,B)$  if and only if $M$ is a projective $A$-module.
\end{Proposition}
\begin{proof} (1) That  $\mathbb{X}$ contains all injective $A$-module follows directly from  the definition of the functor $\mathrm{Ext}^i_A(M,-)$. Since $\mathrm{Ext}^i_A(M,-)$ commutes with products and particularly, finite direct sums, $\mathbb{X}$ is closed under direct summands. Finally, that $\mathbb{X}$ is closed under extensions and the co-kernels of injective homomorphisms follows from the Long Exact Sequence Theorem.

(2) Let $\mathcal{X}$ be an $M^{\bot}$-co-resolving subcategory of $A$-Mod. Take $X\in \mathcal{X}$ and assume that $$0\rightarrow X\rightarrow E_0\xrightarrow{f_0}E_1\xrightarrow{f_1}E_2\xrightarrow{f_2}E_3\rightarrow \cdots$$
 is an injective resolution of $X$. Since $\mathcal{X}$ is co-resolving, it follows that   $X_i:= \mathrm{Im}(f_i)\in \mathcal{X}$ for $i=0,1,2,\dots$. Hence, $$\mathrm{Ext}^i_A(M,X)\cong \mathrm{Ext}_A^{i-1}(M,X_0)\cong \cdots\cong \mathrm{Ext}_A^{1}(M,X_{i-2})=0 $$ for all $i\geq 1$, and so $\mathcal{X}\subseteq \mathbb{X}$, as desired.

(3) It is clear from (2).

(4) By (3), $\mathcal{E}(A,M,B)$ is co-resolving if and only if $\mathbb{X}=A\text{-Mod}$. The latter means   $\Ext^i_A(M,X)=0$ for any $X\in A\mathrm{\text{-}Mod}$ and for $i\geq 1$, which is equivalent to saying that $M$ is a projective $A$-module.
 \end{proof}

\begin{Theorem}\label{frobenius} Suppose   $\Hom_A(M,\mathcal{X})\subseteq \mathcal{Y}$ and that  $\mathcal{E}(\mathcal{X},M,\mathcal{Y})$ is a co-resolving subcategory of  $\mathrm{Rep}_h(A,M,B)$. Then the following statements are equivalent:

$\mathrm{(1)}$ $\mathcal{E}(\mathcal{X},M,\mathcal{Y})$ is a Frobenius category;

$\mathrm{(2)}$ Both $\mathcal{X}$ and $\mathcal{Y}$ are Frobenius category and $\Hom_A(M,\mathcal{I}(\mathcal{X}))\subseteq I(\mathcal{Y})$;

$\mathrm{(3)}$ Both $\mathcal{X}$ and $\mathcal{Y}$ are Frobenius category and $\Hom_A(M,\mathcal{P}(\mathcal{X}))\subseteq P(\mathcal{Y})$.

\end{Theorem}

\begin{proof} We first  observe the following fact: if $\mathcal{C}$ is a co-resolving subcategory of $R$-Mod, where $R$ is a ring, then  the injective objects of $\mathcal{C}$ are the same as those of $R$-Mod. To see this, let $U$ be an injective object of $\mathcal{C}$. Then there is a short exact sequence in $\mathcal{C}$:
$0\rightarrow U\xrightarrow{f} V\rightarrow N\rightarrow 0$ such that $V$ is an injective object of $R$-Mod. It is not hard to see that this sequence splits and so $U$ is a direct summand of $V$. From this, it follows that  $U$ is an injective object of $R$-Mod. The converse is trivial.

$\mathrm{(2)}\Longleftrightarrow \mathrm{(3)}$  It is clear from the definition of a Frobenius category.

In the rest part of this proof, we write $\mathcal{E}$ for $\mathcal{E}(\mathcal{X},M,\mathcal{Y})$.

$\mathrm{(1)}\Longrightarrow \mathrm{(2)}$  If $Y\in \mathcal{I(Y)}$, then $(0,Y)_0\in \mathcal{I(E)}$ by  Theorem~\ref{in1}. Since $\mathcal{I(E)}=\mathcal{P(E)}$, it follows that $(0,Y)_0\in \mathcal{P(E)}$ and this implies $Y\in \mathcal{P(Y)}$ and so $\mathcal{I(Y)}\subseteq \mathcal{P(Y)}$ by Corollary~\ref{projective2}. Similarly, we have $\mathcal{I(Y)}\supseteq\mathcal{P(Y)}$.

Let $X\in \mathcal{I}(\mathcal{X})$.  Then $(X,\Hom_A(M,X))_1\in \mathcal{I(E)}$ by   Theorem~\ref{in1}.  It follows that  $(X,\Hom_A(M,X))\in\mathcal{ P(E)}$, and so $X\in \mathcal{P(X)}$ and $\Hom_A(M,X)\in \mathcal{P(Y)}$. Hence  $\Hom_A(M,\mathcal{I}(\mathcal{X}))\subseteq \mathcal{I(Y)}$ and $\mathcal{I}(\mathcal{X})\subseteq \mathcal{P}(\mathcal{X})$. Conversely, we take $X\in \mathcal{P(X)}$.  Since $\mathcal{Y}$ has enough projective objects by Theorem~\ref{enough}, there is  a short exact sequence $0\rightarrow K\rightarrow P\xrightarrow{\varphi} \Hom_A(M,X)\rightarrow 0$ in $\mathcal{Y}$ such that $P\in \mathcal{P(Y)}$. It follows that $(X,P)_{\varphi}\in \mathcal{P(E)}=\mathcal{I(E)}$ and so $X\in \mathcal{I(X)}$. These facts together with Theorem~\ref{enough} imply both $\mathcal{X}$ and $\mathcal{Y}$ are Frobenius categories.

$\mathrm{(2)}\Longrightarrow \mathrm{(1)}$  If $(X,Y)_{\varphi}\in \mathcal{P(E)}$, then $X\in \mathcal{P(X)}=\mathcal{I(X)}$ and $Y\in \mathcal{P(Y)}=\mathcal{I(Y)}$ by Corollary~\ref{projective2}.    Since $\Hom_A(M,X)\in \mathcal{P(Y)}$, the short exact sequence $0\rightarrow \Ker{\varphi}\rightarrow Y\xrightarrow{\varphi} \Hom_A(M,X)\rightarrow 0$ splits and it follows that $\Ker(\varphi)\in \mathcal{I(Y)}$. Hence $(X,Y)_{\varphi}\in \mathcal{I(E)}$.

Conversely, if $(X,Y)_{\varphi}\in \mathcal{I(E)}$ then $\Ker(\varphi)\in \mathcal{I(Y)}$ and $X\in \mathcal{I(X)}$. From this it follows that $0\rightarrow \Ker{\varphi}\rightarrow Y\xrightarrow{\varphi} \Hom_A(M,X)\rightarrow 0$ splits and thus  $Y=\Ker(\varphi)\oplus \Hom_A(M,X)\in \mathcal{I(Y)}=\mathcal{P(Y)}$. Consequently $(X,Y)_{\varphi}\in \mathcal{P(E)}$ and so $\mathcal{I(E)}=\mathcal{P(E)}$. Now this fact together with Theorem~\ref{enough}  implies $\mathcal{E}$ is a Frobenius category.
\end{proof}

In the final part of this section  we always  assume  $\mathcal{E}(\mathcal{X},M,\mathcal{Y})$ is a co-resolving subcategory of  $\mathrm{Rep}_h(A,M,B)$ and that  $\Hom_A(M,\mathcal{X})\subseteq \mathcal{Y}$.
Under these assumptions, if $(X,Y)_{\varphi}\in \mathcal{E}(\mathcal{X},M,\mathcal{Y})$, then $Y\in \mathcal{Y}$. Hence, we now regard $\mathbf{Q}$ as a functor from $\mathcal{E}(\mathcal{X},M,\mathcal{Y})$ to $\mathcal{Y}$. We establish the following  recollement of triangulated categories  when $\mathcal{E}(\mathcal{X},M,\mathcal{Y})$ is a Frobenius category in a similar way as in \cite[Theorem 1.3]{XZZ}.

\begin{equation}  \label{recollement}\tag{Fig 2} \xymatrix{\mathrm{St}\ \mathcal{Y}\ar@<0.1ex>[rrr]|-{\overline{\mathbf{q}}} &&&  \mathrm{St}\ \mathcal{E}(\mathcal{X},M,\mathcal{Y})\ar@<1.6ex>@/^/[lll]^-{\overline{\mathbb{Q}}}\ar@<-1.8ex>@/_/[lll]_-{\overline{\mathbf{Q}}}\ar@<0.1ex>[rrr]|{\overline{\mathbf{P}}}& & &\mathrm{ \mathrm{St}}\ \mathcal{X} \ar@<1.6ex>@/^/[lll]^-{\overline{\mathbf{p}}}\ar@<-1.8ex>@/_/[lll]_-{},}
\end{equation}

\begin{Lemma}\label{adjoint}  $(\mathbf{Q,q})$, $(\mathbf{q},\mathbb{Q})$ and $(\mathbf{P,p})$ are adjoint pairs.
\end{Lemma}
\begin{proof}  Let $Y_1\in \mathcal{Y}$ and $(X,Y)_{\varphi}\in \mathcal{E}(\mathcal{X},M,\mathcal{Y})$.
Then \begin{equation*}\begin{aligned}\Hom((X,Y)_{\varphi},\mathbf{q}(Y_1))\cong \Hom((X,Y)_{\varphi},(0,Y_1)_0)\\\cong\Hom(Y,Y_1)\cong \Hom(\mathbf{Q}(X,Y)_{\varphi},Y_1)\end{aligned}
\end{equation*}

 and \begin{equation*}\begin{aligned}\Hom(\mathbf{q}(Y_1), (X,Y)_{\varphi})\cong \Hom((0,Y_1)_0, (X,Y)_{\varphi})\\=\{g\in \Hom(Y_1,Y)|\varphi g=0 \}\cong \Hom(Y_1,\Ker(\varphi))\\=\Hom(Y_1,\mathbb{Q}(X,Y)_{\varphi})\end{aligned}
\end{equation*}
This shows that $(\mathbf{Q,q})$ and $(\mathbf{q}, \mathbb{Q})$ are adjoint pairs. Similarly, $(\mathbf{P,p})$ is an adjoint pair.

\end{proof}

\begin{Lemma} \label{quotient} If $\mathcal{E}(\mathcal{X},M,\mathcal{Y})$ is a Frobenius category then there is a triangulated equivalence:  $\mathrm{St}\ \mathcal{E}(\mathcal{X},M,\mathcal{Y}) /\mathbf{\overline{q}}( \mathrm{St} \ \mathcal{Y})\cong  \mathrm{St}\ \mathcal{X}$.
\end{Lemma}
\begin{proof} Denote $\mathcal{E}(\mathcal{X},M,\mathcal{Y})$ by $\mathcal{E}$. We claim that $(X,Y)_{\varphi}$ and $(0,Y)_0$ are isomorphic in  $\mathrm{St}\ \mathcal{E}$ for any $(X,Y)_{\varphi}\in \mathcal{E}$ with $X\in \mathcal{I(X)}$. Let $0\rightarrow \Ker(\varphi)\xrightarrow{f} K\rightarrow N\rightarrow 0$ be a short exact sequence in $\mathcal{Y}$ with $K\in \mathcal{I(Y)}$. Then we obtain  the following commutative diagram:
\begin{equation*} \xymatrix{0\ar[r]&\Ker(\varphi)\ar[r]\ar[d]^{f} & Y\ar[r]^(.35){\varphi}\ar[d]&\Hom_A(M,X)\ar@{=}[d]\ar[r]&0\\
0\ar[r]& K\ar[r] &Q\ar[r]^(0.35){\varphi_1}&\Hom_A(M,X)\ar[r]&0,}
\end{equation*} where the left square is a pushout. This gives rise to a short exact sequence in $\mathcal{E}$:
$$0\rightarrow(X,Y)_{\varphi}\rightarrow (X,Q)_{\varphi_1}\rightarrow (0,N)_0\rightarrow 0. $$
The fact that $(X,Q)_{\varphi_1}\in \mathcal{I(E)}$ together with the short exact sequence   $0\rightarrow(0,Y)_{0}\rightarrow (0,Q)_{0}\rightarrow (0,N)_0\rightarrow 0$ implies $(X,Y)_{\varphi}\cong (0,Y)_0$ in $\mathrm{St}\ \mathcal{E}$, as claimed.

  By Lemma~\ref{tri}, the functor  $\mathbf{P}:\mathcal{E}(\mathcal{X},M,\mathcal{Y})\rightarrow \mathcal{X}$ induces a triangulated functor $\mathbf{\overline{P}}: \mathrm{St}\ \mathcal{E}(\mathcal{X},M,\mathcal{Y})\rightarrow \mathrm{St}\ \mathcal{X}$,
 whose Kernel is $\mathbf{\overline{q}}( \mathrm{St} \ \mathcal{Y})$ by the claim above.
\end{proof}

For a subcategory $\mathcal{S}$ of an additive category $\mathcal{T}$, we set $\mathcal{S}^{\bot}:=\{X\in \mathcal{T}|\Hom(S,X)=0, \forall S\in \mathcal{S}\}$ and $^{\bot}\mathcal{S}:=\{X\in \mathcal{T}|\Hom(X,S)=0, \forall S\in \mathcal{S}\}$. The following result is known essentially, see \cite{CPS1, CPS2}.

\begin{Lemma} \label{leftright} Let $\mathcal{S}$ be a thick triangulated subcategory of a triangulated category $\mathcal{T}$. Then the following are equivalent:

$\mathrm{(1)}$ The quotient functor $\mathbf{Q}:\mathcal{T}\rightarrow \mathcal{T/S}$ has a right (resp. left) adjoint;

$\mathrm{(2)}$ $(\mathcal{S},\mathcal{S}^{\bot})$ (resp. $(\mathcal{S}^{\bot}, \mathcal{S})$) is a torsion pair of $\mathcal{T}$;

$\mathrm{(3)}$ The embedding functor $\mathbf{i}: \mathcal{S}\rightarrow \mathcal{T}$ has a  (resp. left) right adjoint.
\end{Lemma}

We also need the following easy result, see e.g. \cite{WL}.

\begin{Lemma} \label{adjointq}Let $(\mathbf{F,G})$ be an adjoint pair, where $\mathbf{F}:\mathcal{C}\rightarrow \mathcal{D}$ and $\mathbf{G}:\mathcal{D}\rightarrow \mathcal{C}$ are additive functors. If $\mathcal{X}$ and $\mathcal{Y}$ are additive subcategories of  $\mathcal{C}$ and $\mathcal{D}$  respectively  such that  $\mathbf{F}\mathcal{X}\subseteq \mathcal{Y}$ and $\mathbf{G}\mathcal{Y}\subseteq \mathcal{X}$, then $(\overline{\mathbf{F}},\overline{\mathbf{G}})$ with $\mathbf{\overline{F}}: \mathcal{C/X}\rightarrow \mathcal{D/Y}$ and $\mathbf{\mathbf{\overline{G}}}: \mathcal{D/Y}\rightarrow \mathcal{C/X}$ is also  an adjoint pair.
\end{Lemma}

Combining the facts above we obtain:

\begin{Theorem} With the same assumptions as well as one of the equivalent statements of Theorem~\ref{frobenius},    the diagram in (\ref{recollement}) is a recollement of triangulated categories.
\end{Theorem}

\begin{proof} In view of  Lemmas~\ref{tri} and \ref{adjointq}, the functors  $\overline{\mathbf{Q}},\overline{\mathbf{q}},\overline{\mathbb{Q}}$, $\overline{\mathbf{p}}$ and $\overline{\mathbf{P}}$ are all triangulated functors and $(\overline{\mathbf{Q}},\ \overline{\mathbf{q}}),\ (\overline{\mathbf{q}}, \overline{\mathbb{Q}})$, $(\overline{\mathbf{p}},\ \overline{\mathbf{P}})$ are adjoint pairs. From these, we obtain the left side of (\ref{recollement}), and it can be completed into (\ref{recollement}) by Lemmas~\ref{quotient} and \ref{leftright}.
\end{proof}

\section{The properties of $\mathcal{M}(\mathcal{X},M,\mathcal{Y})$ \label{5}}

The theory of $\mathcal{M}(\mathcal{X},M,\mathcal{Y})$ is perfectly dual to that of   $\mathcal{E}(\mathcal{X},M,\mathcal{Y})$, whether on their results or on their proofs.  Hence, we will only present  the results on $\mathcal{M}(\mathcal{X},M,\mathcal{Y})$, but  omitting all the proofs in this section.

\begin{Lemma}  If both $\mathcal{X}$ and $\mathcal{Y}$ are closed under extensions, then $\mathcal{M}(\mathcal{X},M,\mathcal{Y})$ is closed under extensions.
\end{Lemma}

\begin{Proposition} \label{injective2} Suppose that $\mathcal{M}(\mathcal{X},M,\mathcal{Y})$ is closed under extensions and that $\mathrm{Tor}_1^B(M,\mathcal{Y})=0$. If either $\mathcal{X}$ has enough projective objects or $M\otimes_B\mathcal{Y}\subseteq \mathcal{X}$, then for any $\begin{bmatrix}U\\V\end{bmatrix}_{\phi}\in \mathcal{M}(\mathcal{X},M,\mathcal{Y})$,  $\begin{bmatrix}U\\V\end{bmatrix}_{\phi}$ is an injective object of $\mathcal{M}(\mathcal{X},M,\mathcal{Y})$ if and only if
$U$ is an injective object of $\mathcal{X}$ and $V$ is  $\mathcal{Y}$-injective.
\end{Proposition}

\begin{Proposition} \label{enoughi}  Suppose that $\mathcal{M}(\mathcal{X},M,\mathcal{Y})$ is closed under extensions and assume further that $\mathrm{Tor}_1^B(M,\mathcal{Y})=0$  and   $M\otimes_B\mathcal{Y}\subseteq \mathcal{X}$. Then $\mathcal{M}(\mathcal{X},M,\mathcal{Y})$ has enough injective objects if and only if both $\mathcal{X}$ and $\mathcal{Y}$ have enough injective objects.
\end{Proposition}

Let $\mathcal{C}$ be an abelian category which has enough projective objects. A full subcategory $\mathcal{D}$ of $\mathcal{C}$ is a {\it resolving} subcategory of $\mathcal{C}$ or just resolving if $\mathcal{P(C)}\subseteq \mathcal{D}$ and $\mathcal{D}$ is closed under direct summands, extensions and the kernels of surjective homomorphisms.

\begin{Theorem} \label{resolving} The following statements are equivalent:

$\mathrm{(1)}$ $\mathcal{M}(\mathcal{X},M,\mathcal{Y})$ is a resolving subcategory of $\mathrm{Rep}(A,M,B)$;

$\mathrm{(2)}$ Both $\mathcal{X}$ and $\mathcal{Y}$ are resolving and $\mathrm{Tor}_1^B(M, \mathcal{Y})=0$.
\end{Theorem}

A full subcategory $\mathcal{Y}$ of $B\text{-}$Mod is called $M^{\bot}$-{\it resolving}   if it is  resolving and $\mathrm{Tor}_1^B(M,Y)=0$ for any $Y\in \mathcal{Y}$.

\begin{Proposition} Set $\mathbb{Y}:=\{Y\in B\mathrm{\text{-}Mod}|\Tor_i^B(M,Y)=0,  \forall i\geq 1\}$. Then

$\mathrm{(1)}$  $\mathbb{Y}$ is  an $M^{\bot}$-resolving subcategory;

$\mathrm{(2)}$  Any $M^{\bot}$-resolving subcategory of $B$-Mod  is included in $\mathbb{Y}$;

$\mathrm{(3)}$ $\mathcal{M}(A,M,\mathbb{Y})$ is the largest resolving subcategory of $\mathrm{Rep}(A,M,B)$ of the form $\mathcal{M}(\mathcal{X},M,\mathcal{Y})$;

$\mathrm{(4)}$ $\mathcal{M}(A,M,B)$ is a resolving  subcategory of $\mathrm{Rep}(A,M,B)$  if and only if $M$ is a flat $B$-module.
\end{Proposition}

\begin{Theorem}\label{frobeniusm} Suppose that $\mathcal{M}(\mathcal{X},M,\mathcal{Y})$ is a resolving subcategory of  $\mathrm{Rep}(A,M,B)$ and that  $M\otimes_B \mathcal{Y}\subseteq \mathcal{X}$. Then the following statements are equivalent:

$\mathrm{(1)}$ $\mathcal{M}(\mathcal{X},M,\mathcal{Y})$ is a Frobenius category;

$\mathrm{(2)}$ Both $\mathcal{X}$ and $\mathcal{Y}$ are Frobenius categories and $M\otimes_B \mathcal{I(Y)}\subseteq \mathcal{I(X)}$;

$\mathrm{(3)}$ Both $\mathcal{X}$ and $\mathcal{Y}$ are Frobenius categories and $M\otimes_B \mathcal{P(Y)}\subseteq \mathcal{P(X)}$.

\end{Theorem}

\begin{Theorem} Under the same assumptions together with one of the equivalent statements of Theorem~\ref{frobeniusm}, we have the following recollement of triangulated categories.
\begin{equation} \tag{Fig 3} \xymatrix{\mathrm{St}\ \mathcal{X}\ar@<0.1ex>[rrr]|-{\overline{\mathbf{p_M}}} &&&  \mathrm{St}\ \mathcal{M}(\mathcal{X},M,\mathcal{Y})\ar@<1.6ex>@/^/[lll]^-{\overline{\mathbf{P_M}}}\ar@<-1.8ex>@/_/[lll]_-{\overline{\mathbb{P}_M}}\ar@<0.1ex>[rrr]|{\overline{\mathbf{Q_M}}}& & &\mathrm{ \mathrm{St}}\ \mathcal{Y} \ar@<1.6ex>@/^/[lll]\ar@<-1.8ex>@/_/[lll]_-{^{\overline{\mathbf{q_M}}}}.}
\end{equation}
\end{Theorem}
\noindent Here, the  functors $\overline{\mathbb{P}_M}$, $\overline{\mathbf{p}_M}$,  $\overline{\mathbf{P}_M}$,   $\overline{\mathbf{Q}_M}$  $\overline{\mathbf{q}_M}$ are induced by the following ones respectively:

 $\mathbb{P}_M: \mathcal{M}(\mathcal{X},M,\mathcal{Y})\rightarrow \mathcal{X}$, $\begin{bmatrix}X\\Y\end{bmatrix}_{\varphi}\longmapsto \Coker(\varphi)$;

$\mathbf{p}_M:\mathcal{X}\rightarrow \mathcal{M}(\mathcal{X},M,\mathcal{Y})$,  $X\longmapsto \begin{bmatrix}X\\0\end{bmatrix}_{0}$;

$\mathbf{P}_M: \mathcal{M}(\mathcal{X},M,\mathcal{Y})\rightarrow \mathcal{X}$, $\begin{bmatrix}X\\Y\end{bmatrix}_{\varphi}\longmapsto X$;

$\mathbf{Q}_M: \mathcal{M}(\mathcal{X},M,\mathcal{Y})\rightarrow \mathcal{Y}$, $\begin{bmatrix}X\\Y\end{bmatrix}_{\varphi}\longmapsto Y$;

$\mathbf{q}_M:\mathcal{Y}\rightarrow \mathcal{M}(\mathcal{X},M,\mathcal{Y})$, $Y\longmapsto \begin{bmatrix}M\otimes_B Y\\Y\end{bmatrix}_{1}$.

\section{Applications to Comma Categories}

Assume that $\mathcal{A}$ and $\mathcal{B}$ are  abelian categories and let $\mathbf{F} : \mathcal{A}\rightarrow \mathcal{B}$ and $\mathbf{G}: \mathcal{B}\rightarrow \mathcal{A}$ be  additive functors.
The comma category $(\mathcal{B}, \mathbf{F}\mathcal{A })$ is defined as follows:

$\mathrm{obj}(\mathcal{B}, \mathbf{F}\mathcal{A } )$: triples $(B, \mathbf{F}A)_f $, where $A \in \mathcal{A}$ , $B \in \mathcal{B}$ and $f: B  \rightarrow \mathbf{F}A\in  \mathrm{Mor}\mathcal{B}$;

$\mathrm{Mor}(\mathcal{B}, \mathbf{F}\mathcal{A })$: pairs $(\alpha,\beta) : (B, \mathbf{F}A)_f  \rightarrow (B', \mathbf{F}A')_{f'} $, where $\alpha\in \Hom_{\mathcal{A}}(A,A')$ and $\beta\in  \Hom_{\mathcal{B}}(B,B')$ and the following diagram is commutative:
\begin{equation*}\xymatrix  {B\ar[rr]^{f}\ar[d]^{\beta}&&\mathbf{F}A\ar[d]^{\mathbf{F}\alpha}\\
B'\ar[rr]^{f'}&&\mathbf{F}A' .}
\end{equation*}

Dually, the comma  category $(\mathbf{G}\mathcal{B }, \mathcal{A})$ is defined as follows:

$\mathrm{obj}(\mathbf{G}\mathcal{B }, \mathcal{A} )$: triples $(\mathbf{G}B,A)_f$, where $A \in \mathcal{A}$ , $B \in \mathcal{B}$ and $f:\mathbf{G}B \rightarrow A \in  \mathrm{Mor}\mathcal{A}$;

$\mathrm{Mor}(\mathbf{G}\mathcal{B }, \mathcal{A})$: pairs $(\alpha,\beta): (\mathbf{G}B,A)_f \rightarrow (\mathbf{G}B',A')_{f'} $, where $\alpha\in \Hom_{\mathcal{A}}(A,A')$ and $\beta\in  \Hom_{\mathcal{B}}(B,B')$ and the following diagram is commutative:
\begin{equation*}\xymatrix  {\mathbf{G}B\ar[d]^{\mathbf{G}\beta}\ar[rr]^{f}&&A\ar[d]^{\alpha}\\
\mathbf{G}B'\ar[rr]^{f'}&&A'. }
\end{equation*}

The following result is well-known, see e.g.  \cite{FGR}.

\begin{Lemma}  {\rm(1)} If $\mathbf{F}$ is left exact, then $(\mathcal{B}, \mathbf{F}\mathcal{A } )$ is an abelian category.

 {\rm(2)}  If $\mathbf{G}$ is right exact, then $(\mathbf{G}\mathcal{B }, \mathcal{A})$ is an abelian category.
\end{Lemma}

In general it is difficult to identify projective objects or injective objects of a comma category. In this section we  point out that some proofs  in previous sections can be transplanted into  the case of comma categories  and this allow us to classify projective objects of $\mathcal{E}(\mathcal{Y},\mathbf{F}\mathcal{X})$ and injective objects of $\mathcal{M}(\mathbf{G}\mathcal{Y},\mathcal{X})$ when $\mathbf{F}$ and $\mathbf{G}$ are  exact functors, where $\mathcal{E}(\mathcal{Y},\mathbf{F}\mathcal{X})$ and $\mathcal{M}(\mathbf{G}\mathcal{Y},\mathcal{X})$ are   exact subcategories of comma categories which are defined below.

\begin{Definition} \label{ME} \em Let $\mathcal{X}$ (resp. $\mathcal{Y}$) be full subcategories of $\mathcal{A}$ (resp. $\mathcal{B}$) closed under isomorphisms and containing $0$. By definition, $\mathcal{E}(\mathcal{Y},\mathbf{F}\mathcal{X})$ is a full subcategories of $(\mathcal{B},\mathbf{F}\mathcal{A})$ whose objects are $(Y,FX)_f$ such that $f$ is surjective, $X\in \mathcal{X}$ and $\Ker(f)\in \mathcal{Y}$;  and $\mathcal{M}(\mathbf{G}\mathcal{Y},\mathcal{X})$ is a full subcategories of $(\mathbf{G}\mathcal{B},\mathcal{A})$ whose objects are $(\mathbf{G}Y,X)_f$ such that $f$ is injective, $Y\in \mathcal{Y}$ and $\Coker(f)\in \mathcal{X}$.

\end{Definition}

 We now present the results of this section.

\begin{Proposition} \label{E} Suppose that $\mathbf{F}$ is left exact and both $\mathcal{X}$ and $\mathcal{Y}$ are closed under extensions.

{\rm (1)} $\mathcal{E}(\mathcal{Y},\mathbf{F}\mathcal{X})$ is closed under extensions;

{\rm (2)} If assume further that  $\mathbf{F}$ is exact and $\mathbf{F}\mathcal{X}\subseteq \mathcal{Y}$,  then

{\rm (i)}  a triple $(Y,\mathbf{F}X)_f$ is a projective object of $\mathcal{E}(\mathcal{Y},\mathbf{F}\mathcal{X})$ if and only if $X$ is a projective object of $\mathcal{X}$ and $Y$ is a projective object of $\mathcal{Y}$;

{\rm (ii)}  $\mathcal{E}(\mathcal{Y},\mathbf{F}\mathcal{X})$ has enough projective objects if and only if both $\mathcal{X}$ and $\mathcal{Y}$ have enough projective objects.
\end{Proposition}

\begin{proof} The proofs of (1), 2(i), and 2(ii) are essentially the same as those of Proposition~\ref{extension}, Corollary~\ref{injective2} and Theorem~\ref{enough}, respectively.
\end{proof}

Dually we have the following result.

\begin{Proposition} \label{M} Suppose that $\mathbf{G}$ is right exact and both $\mathcal{X}$ and $\mathcal{Y}$ are closed under extensions.

{\rm (1)} $\mathcal{M}(\mathbf{G}\mathcal{Y},\mathcal{X})$ is closed under extensions;

{\rm (2)} If assume further that  $\mathbf{G}$ is exact and $\mathbf{G}\mathcal{Y}\subseteq \mathcal{X}$,  then

{\rm (i)}  a triple $(\mathbf{G}Y,X)_f$ is a projective object of $\mathcal{M}(\mathbf{G}\mathcal{Y},\mathcal{X})$ if and only if $X$ is an injective object of $\mathcal{X}$ and $Y$ is an injective object of $\mathcal{Y}$;

{\rm (ii)} $\mathcal{M}(\mathbf{G}\mathcal{Y},\mathcal{X})$ has enough injective objects if and only if both $\mathcal{X}$ and $\mathcal{Y}$ have enough injective objects.
\end{Proposition}

\section*{Appendix}

In this appendix, we give a proof of Lemma~\ref{leftright}.
Some preparations are needed.

\begin{Lemma}  \label{faithful} Let $\mathcal{S}$ be a thick triangulated subcategory of a triangulated category $\mathcal{T}$ and $\mathbf{Q}: \mathcal{T}\rightarrow \mathcal{T/S}$ the quotient functor. If $\mathbf{Q}$ has a right (resp. left) adjoint $\mathbf{R}$ (resp. $\mathbf{L}$), then $\mathbf{R}$ (resp. $\mathbf{L}$) is fully faithful.
\end{Lemma}
\begin{proof} This follows from  \cite[Proposition 1.1.2, Lemma 3.2.1 and Proposition 3.2.2]{H}.
\end{proof}

\begin{Lemma} \label{local} Let $\mathbf{Q}: \mathcal{T}\rightarrow \mathcal{T/S}$ be the quotient functor.

{\em (1) } If $T\in \mathcal{T}$ and $X\in \mathcal{S}^{\bot}$, then $\Hom_{\mathcal{T}}(T,X)\cong \Hom_{\mathcal{T/S}}(\mathbf{Q}(T),\mathbf{Q}(X))$.

{\em (2) }   If $T\in \mathcal{T}$ and $X\in ^{\bot}\mathcal{S}$, then $\Hom_{\mathcal{T}}(X,T)\cong \Hom_{\mathcal{T/S}}(\mathbf{Q}(X),\mathbf{Q}(T))$.

\end{Lemma}
\begin{proof} This follows from  \cite[Lemmas 1.2.6, 3.2.1 and Proposition 3.2.2]{H}
\end{proof}

We now prove the right case of Lemma~\ref{leftright}. (The proof of the left case is dual, and thus omitted.)
\begin{proof}
 $\mathrm{(1)}\Rightarrow \mathrm{(2)}$ Let  $\mathbf{R}$ be the right adjoint to  $\mathbf{Q}$ and  $\eta:1_{\mathcal{T}}\rightarrow \mathbf{RQ}$ the unit of  $(\mathbf{Q,R})$. Then,  for any  $T\in \mathcal{T}$,   the morphism  $T\xrightarrow{\eta_T}\mathbf{RQ}(T)$  fits into an exact triangle of $\mathcal{T}$: $$S\rightarrow T\xrightarrow{\eta_T} \mathbf{RQ}(T)\rightarrow S[1].$$  This gives rise to   an exact triangle $\mathbf{Q}(S)\rightarrow \mathbf{Q}(T)\xrightarrow{\mathbf{Q} (\eta_T)} \mathbf{QRQ}(T)\rightarrow \mathbf{Q}(T)[1]$ of $\mathcal{T/S}$.     Since $\mathbf{R}$ is fully faithful by Lemma~\ref{faithful},  $\mathbf{Q}(\eta_T)$ is an isomorphism. From this it follows that  $\mathbf{Q}(S)=0$ and so $S\in \mathcal{S}$.  On the other hand, we have $$\mathrm{Hom}_{\mathcal{T}}(Y, \mathbf{RQ}(T))\cong \mathrm{Hom}_{\mathcal{T/S}}(\mathbf{Q}(Y), \mathbf{Q}(T))=0$$ for any $Y\in \mathcal{S}$ and this implies $\mathbf{RQ}(T)\in \mathcal{S}^{\bot}$. Hence $(\mathcal{S},\mathcal{S}^{\bot})$ is a torsion pair.

$\mathrm{(3)}\Rightarrow \mathrm{(2)}$  can be proved in a similar way as in  $\mathrm{(1)}\Rightarrow \mathrm{(2)}$

$\mathrm{(2)}\Rightarrow \mathrm{(1)}$  Let $\mathbf{Q}': \mathcal{S}^{\bot}\rightarrow \mathcal{T/S}$ be the restriction functor of $\mathbf{Q}$. Then $\mathbf{Q}'$ is fully faithful by Lemma~\ref{local}. Since $(\mathcal{S},\mathcal{S}^{\bot})$ is a torsion pair,  there is for any  $T\in \mathcal{T}$ an exact triangle $S\rightarrow T\rightarrow Z\rightarrow Y[1]$ such that  $S\in \mathcal{S}$ and $Z\in \mathcal{S}^{\bot}$. Note that $\mathbf{Q}(T)\cong \mathbf{Q}(Z)$,  it follows that $\mathbf{Q}'$ is {\it dense} (i.e.,  there is for any $X\in \mathcal{T/S}$ an object $Z\in \mathcal{S}^{\bot}$ such that $X\cong \mathbf{Q}'(Z)$), and so it is an equivalence.

Now set $\mathbf{R}=\mathbf{jF}$, where $\mathbf{j}:\mathcal{S}^{\bot}\rightarrow \mathcal{T}$ is the embedding functor and $\mathbf{F}$ is the quasi-inverse to $\mathbf{Q}'$.  Then, for any $T\in \mathcal{T}$ and $X\in \mathcal{T/S}$, we have
 \begin{equation*}\begin{aligned}\Hom_{\mathcal{T}}(T,\mathbf{R}(X))&=\Hom_{\mathcal{T}}(T,\mathbf{F}(X))\\
&\cong\Hom_{\mathcal{T/S}}(\mathbf{Q}(T),\mathbf{QF}(X))\cong \Hom_{\mathcal{T/S}}(\mathbf{Q}(T),X).
\end{aligned}
\end{equation*}
Here the second isomorphism follows from Lemma~\ref{local}. Consequently, $\mathbf{R}$ is the right adjoint to $\mathbf{Q}$.

$\mathrm{(2)}\Rightarrow \mathrm{(3)}$ For any $T\in \mathcal{T}$, let  $\mathbf{u}(T)\xrightarrow{\phi_T} T\rightarrow \mathbf{v}(T)\rightarrow \mathbf{u}(T)[1]$ be an exact triangle such that $\mathbf{u}(T)\in \mathcal{S}$ and $\mathbf{v}(T)\in \mathcal{S}^{\bot}$.

 Given $S,T\in \mathcal{T}$. Note that $\Hom(\mathbf{u}(T), \phi_S): \Hom_{\mathcal{T}}(\mathbf{u}(T),\mathbf{u}(S))\rightarrow \Hom_{\mathcal{T}}(\mathbf{u}(T),S)$ is an isomorphism, there exists for any morphism $f:T\rightarrow S$ a unique morphism $\mathbf{u}(f):\mathbf{u}(T)\rightarrow \mathbf{u}(S)$ such that the the following diagram is commutative:
\begin{equation*}  \label{CD3}\xymatrix{
    &   \mathbf{u}(T)\ar[r]^{\phi_T} \ar[d]_{\mathbf{u}(f)} & T \ar[d]_{f}  \\
    &    \mathbf{u}(S)\ar[r]^{\phi_S} & S   }
\end{equation*}
This implies that $\mathbf{u}(T)$ is uniquely determined up to isomorphism  by $T$ and that $\mathbf{u}$ is a functor from $\mathcal{T}$ to $\mathcal{S}$. By an easy check, $\mathbf{u}$ is the right adjoint to $\mathbf{i}$.
\end{proof}
\noindent{\bf Acknowledgement}:  This project is supported by  NSFC (No. 11971338)

\end{document}